\documentclass[a4paper,12pt]{article}
\usepackage[english]{babel}
\usepackage{amsmath,amssymb,amsfonts,amscd}
\usepackage{amsthm}
\newtheorem {theorem}{Theorem}
\newtheorem{remark}[theorem]{Remark}
\newtheorem{corollary}[theorem]{Corollary}
\newcommand{\R}{\mathbb R}
\newcommand{\N}{\mathbb N}
\begin{document}

\title{On the Dirichlet problem associated with Dunkl Laplacian}
\author{ Ben Chrouda Mohamed\\
 {\small High Institute of Informatics and Mathematics}\\
 {\small 5000  Monastir, Tunisia}\\
 {\small E-mail: benchrouda.ahmed@gmail.com}}
%\author{}
%\ead{khalifa.elmabrouk@fsm.rnu.tn}
%\address{Department of Mathematics, Faculty of Sciences of Monastir, 5019 Monastir, Tunisia}
%\address[label2]{Department of Mathematics, High School of Sciences and Technology, 4011 Hammam Sousse, Tunisia}

\maketitle
\begin{abstract} This paper is devoted to the study of the Dirichlet problem associated with the Dunkl Laplacian $\Delta_k$.
We establish, under some condition on a bounded domain $D$ of $\R^d$, the existence of a unique continuous function $h$ on $\R^d$ such that  $\Delta_kh=0$ on $D$ and $h=f$ on $\R^d\setminus D$ the complement of $D$ in $\R^d$, where
 the function $f$ is asumed to be continuous. We also give an analytic formula characterizing the solution $h$.
\end{abstract}
\section{Introduction}
Let $R$ be a root system in $\R^d$, $d\geq1$, and we fix a positive subsystem $R_+$ of $R$ and a nonnegative multiplicity function $k:R\to\R_+$.
For every $\alpha\in R$, let $H_\alpha$ be the hyperplane  orthogonal to $\alpha$ and $\sigma_\alpha$ be the reflection with respect to~$H_\alpha$, that is, for every~$x\in\R^d$,
$$
\sigma_\alpha x=x-2 \frac{\langle x,\alpha\rangle}{|\alpha|^2}\alpha
$$
where $\langle\cdot,\cdot\rangle$ denotes the Euclidean inner product of $\R^d$.
The Dunkl Laplacian $\Delta_k$ is defined \cite{dunkl1},  for  $f\in C^2(\R^d)$, by
$$
\Delta_kf(x)=\Delta f(x)+2\sum_{\alpha\in R_+}k(\alpha)\left(\frac{\langle\nabla f(x),\alpha\rangle}{\langle\alpha,x\rangle}-\frac{|\alpha|^2}2\frac{f(x)-f(\sigma_\alpha x)}{\langle\alpha,x\rangle^2}\right),
$$
where $\nabla$ denotes the  gradient on~$\R^d$. Obviously, $\Delta_k=\Delta$ when $k\equiv0$.

Given a bounded open subset $D$ of $\R^d$, we consider the following Dirichlet problem~:
\begin{equation}\label{ddp}
\displaystyle\left\{
\begin{array}{rcll}
 \Delta_kh & = & 0 & \mbox{on }\;D, \\
  h & = & f& \mbox{on }\; \R^d\setminus D,
\end{array}
\right.
\end{equation}
where
  $f$ is a continuous function on $\R^d\setminus D$. When $D$ is invariant under all reflections $\sigma_\alpha$, it was shown in \cite{mbke}, using probabilistic  tools from potential theory, that  there exists a unique continuous function $h$ on $\R^d$, twice differentiable on~$D$ and such that both equations in~(\ref{ddp}) are pointwise fulfilled.  In this paper, we shall investigate problem (\ref{ddp}) for a bounded domain $D$  which is not invariant. Let $D$ be a bounded open set such that its closure $\overline{D}$ is in some Domain of $\R^d\setminus\cup_{\alpha\in R_+}H_\alpha$.
We mean by a solution of problem (\ref{ddp}), every function $h : \R^d\to\R$ which is continuous on $\R^d$ such that $h=f$ on $\R^d\setminus D$ and
%\begin{equation}\label{wddp}
$$
\int_{\R^d}h(x)\Delta_k\varphi(x) w_k(x) dx =0\quad \textrm{ for every }\; \varphi \in C^\infty_c(D),
$$
%\end{equation}
where $C^\infty_c(D)$ denotes the space of infinitely differentiable functions on $D$  with compact support  and $w_k$ is the invariant weight function defined on $\R^d$ by
$$
w_k(x)=\prod_{\alpha\in
R_+}\langle x,\alpha\rangle^{2k(\alpha)}.
$$
The set $D$ is called $\Delta_k$-regular if, for every continuous function $f$ on $\R^d\setminus D$, problem (\ref{ddp}) admits one and only one solution; this solution will be denoted by $H_D^{\Delta_k}f$.
By transforming  problem (\ref{ddp})  to a  boundary value problem associated with Schr\"{o}dinger's operator $\Delta-q$, we show that $D$ is $\Delta_k$-regular provided it is $\Delta$-regular. We also give an analytic formula characterizing the solution $H_D^{\Delta_k}f$ (see Theorem \ref{t1} below). We derive from this formula that, for every $x\in D$, $H_D^{\Delta_k}f(x)$ depends only on the values of $f$ on $\cup_{\alpha\in R_+}\sigma_\alpha(D)$ and on $\partial D$ the Euclidean boundary of $D$.
If, in addition, we assume that $f$ is locally H\"{o}lder continuous on $\cup_{\alpha\in R_+}\sigma(D)$ then  $H_D^{\Delta_k}f$ is continuously twice differentiable on $D$ and therefore the first equation in~(\ref{ddp}) is  fulfilled by $H_D^{\Delta_k}f$ not only in the sense of distributions but also pointwise.

It was shown in \cite{kh,hmkt} that the operator $\Delta_k$ is hypoelliptic on all invariant open subset $D$ of $\R^d$. However, if $D$ is not invariant, the question whether $\Delta_k$ is hypoelliptic on $D$ or not remaind open. For $\Delta_k$-regular open set $D$, we show that if $D$ is not invariant then $\Delta_k$ is not hypoelliptic in $D$. Hence the condition " $D$ is invariant" is necessary and sufficient for the hypoellipticity of  $\Delta_k$ on $D$.
\section{Main results}
We first present various facts on the Dirichlet boundary value problem associated with Schr\"{o}dinger's operator which are needed for our approach.
We refer to \cite{abwhhh,kczz} for  details.
Let $G$ be the Green function on $\R^d$, but without the constant factors :
$$
G(x,y)=
\left\{
  \begin{array}{ll}
    |x-y|^{2-d} & \hbox{if}\; d\geq 3; \\
    \ln\frac{1}{|x-y|} & \hbox{if}\; d=2; \\
    |x-y| & \hbox{if}\; d=1.
  \end{array}
\right.
$$
Let $D$ be a bounded domain of $\R^d$ and let $q\in J(D)$ the Kato class on $D$, i.e., $q$ is a Borel measurable function on $\R^d$ such that $G(1_D|q|)$ the Green potential of $1_D|q|$ is continuous on $\R^d$.
Note that the Kato class $J(D)$ contains all bounded Borel measurable functions on $D$.
Assume that  $D$ is $\Delta$-regular. Then, for every continuous function $f$ on  $\partial D$, there exists a unique continuous function $h$ on  $\overline{D}$  such that $h=f$ on $\partial D$ and
\begin{equation}\label{shr}
\int h(x)(\Delta-q)\varphi(x) dx=0 \quad \textrm{ for every}\; \varphi\in C^\infty_c(D).
\end{equation}
In the sequel, we denote $H_D^{\Delta-q}f$ the unique continuous extension on $\overline{D}$ of $f$ which satisfies the Schr\"{o}dinger's equation (\ref{shr}).
Let $G_D^\Delta$ and $G_D^{\Delta-q}$ denotes, respectively, the Green potential operator in $D$ of  $\Delta$ and  $\Delta-q$. The operator $G_D^{\Delta-q}$ acts as a right inverse of the Schr\"{o}dinger's operator $-(\Delta-q)$, i.e., for every Borel bounded function $g$ on $D$, we have
$$
\int G_D^{\Delta-q}g(x)(\Delta-q)\varphi(x) dx= -\int g(x)\varphi(x) dx \quad \textrm{ for every}\; \varphi\in C^\infty_c(D).
$$
Then the unique continuous function $h$ on $\overline{D}$ such that $h=f$ on $\partial D$ and
\begin{equation}\label{shrg}
\int h(x)(\Delta-q)\varphi(x) dx= -\int g(x)\varphi(x) dx \quad \textrm{ for every}\; \varphi\in C^\infty_c(D)
\end{equation}
is given, for $x\in D$, by
\begin{equation}\label{ss}
h(x)= H_D^{\Delta-q}f(x)+G_D^{\Delta-q}g(x).
\end{equation}
The function $G_D^{\Delta-q}g$ is continuous on $\overline{D}$, vanishing on $\R^d\setminus D$ and, for every $x\in D$,
\begin{equation}\label{vgd}
G_D^{\Delta-q}g(x)=G_D^\Delta g(x)- G_D^\Delta(qG_D^{\Delta-q}g)(x).
\end{equation}
Moreover, if, in addition, we assume  that $q\in C^\infty(D)$ then, proceeding by induction, it follows from (\ref{vgd}) that $G_D^{\Delta-q}g\in C^n(D)$ if and only if $G_D^\Delta g\in C^n(D),\; n\in \N$.

Now we are ready to establish our first main result giving a characterization of solutions of the Dirichlet boundary value problem associated with the Dunkl Laplacian $\Delta_k$.

\begin{theorem}\label{t1}
Let $D$ be a bounded open set such that $\overline{D}$ is in some Domain of $\R^d\setminus\cup_{\alpha\in R_+}H_\alpha$. If $D$ is $\Delta$-regular then $D$ is $\Delta_k$-regular. Moreover, for every continuous function $f$ on $\R^d\setminus D$ and for every $x\in D$,
\begin{equation}\label{solhar}
H_D^{\Delta_k}f(x)=\frac{1}{\sqrt{w_k(x)}}\left(H_D^{\Delta-q}(f\sqrt{w_k})(x)+G_D^{\Delta-q}\left(\sqrt{w_k}Nf\right)(x)\right),
\end{equation}
where $q$ and $Nf$ are the functions defined, for $x\in D$, by
$$
q(x):=\sum_{\alpha\in R_+}\left(\frac{|\alpha|k(\alpha)}{\langle x,\alpha\rangle}\right)^2
$$
and
$$
Nf(x):=\sum_{\alpha\in R_+}\frac{|\alpha|^2k(\alpha)}{\langle x,\alpha\rangle^2}f(\sigma_\alpha x).
$$
\end{theorem}

\begin{proof}
Let $f$ be a continuous function on $\R^d\setminus D$. We intend to prove existence and uniqueness of a continuous function $h$ on $D$ such that $h=f$ on $\R^d\setminus D$ and
\begin{equation}\label{har}
\int h(x)\Delta_k\varphi(x) w_k(x) dx=0 \quad \textrm{ for every}\; \varphi\in C^\infty_c(D).
\end{equation}
It is clear that
$$
\nabla\left(\sqrt{w_k}\right)(x)= \sqrt{w_k(x)}\sum_{\alpha\in R_+}\frac{k(\alpha)}{\langle x,\alpha\rangle}\alpha.
$$
Then, using the fact that \cite{dunkl1}
$$
\sum_{\alpha, \beta\in R_+}k(\alpha)k(\beta)\frac{\langle\alpha , \beta\rangle}{\langle x , \alpha\rangle\;\langle x , \beta\rangle}=\sum_{\alpha\in R_+}\frac{|\alpha|^2k^2(\alpha)}{\langle x,\alpha\rangle^2},
$$
direct computation  shows that
$$
\Delta\left(\sqrt{w_k}\right)(x)= \sqrt{w_k(x)}\sum_{\alpha\in R_+}|\alpha|^2\frac{k^2(\alpha)-k(\alpha)}{\langle x,\alpha\rangle^2}.
$$
Thus, for every $\varphi\in C^\infty_c(D)$,
\begin{eqnarray*}
\Delta\left(\varphi\sqrt{w_k}\right)(x)&=& \sqrt{w_k(x)}\left(\Delta\varphi(x)+ 2\sum_{\alpha\in R_+}k(\alpha)\left(\frac{\langle\nabla \varphi(x),\alpha\rangle}{\langle\alpha,x\rangle}-\frac{|\alpha|^2}{2}\frac{\varphi(x)}{\langle\alpha,x\rangle^2}\right)\right)\\
   & & +\; q(x)\varphi(x)\sqrt{w_k(x)},
\end{eqnarray*}
and thereby
\begin{equation}\label{ddk}
\sqrt{w_k(x)}\Delta_k\varphi(x)=\left(\Delta\left(\varphi\sqrt{w_k}\right)(x)- q(x)\varphi(x)\sqrt{w_k(x)}\right)+ \sqrt{w_k(x)}N\varphi(x).
\end{equation}
Since the map $ \varphi\to\varphi\sqrt{w_k}$ is invertible on the space $C^\infty_c(D)$ and the function $x\to\frac{w_k(x)}{\langle x,\alpha\rangle^2}$ is invariant under the reflection $\sigma_\alpha$, equation (\ref{har}) is equivalent to the following Schr\"{o}dinger's equation : For every $\psi\in C^\infty_c(D)$,
$$
\int h(x)\sqrt{w_k(x)}\left(\Delta-q\right)\psi(x) dx= -\int \sqrt{w_k(x)}Nf(x)\psi(x) dx.
$$
Finally, since $q$ is bounded on $D$ and therefore is in $J(D)$, the statements follow from (\ref{shrg}) and (\ref{ss}).
\end{proof}

To construct a $\Delta$-regular set $D$, it suffices to choose $D$ such that its Euclidean boundary $\partial D$  satisfies the the geometric assumption known as " cone condition", i.e.,   for every $z\in \partial D$ there exists a cone $C$ of vertex $z$ such that $C\cap B(z,r)\subset \R^d\setminus D$ for some $r>0$, where $B(z,r)$ is the ball of center $z$ and radius $r$  (see, for example, \cite{kczz}).
\begin{remark}\rm
Note that, in order to obtain $q\in J(D)$, the hypothesis of the above theorem "$\overline{D}\subset\R^d\setminus\cup_{\alpha\in R_+}H_\alpha$" is nearly optimal. Indeed, assume that there exists a cone $C_z$ of vertex $z\in \overline{D}\cap H_\alpha$ for some $\alpha\in R_+$ with $k(\alpha)\neq 0$ such that $C_z^r := C_z\cap B(z,r)\subset D$ for some $r>0$. Then,
\begin{eqnarray*}
G(1_Dq)(z) & \geq & |\alpha|^2k^2(\alpha)\int_{C_z^r}G(z,y)\frac{1}{\langle y,\alpha\rangle^2} dy\\
& = & |\alpha|^2k^2(\alpha)\int_{C_z^r}G(z,y)\frac{1}{\langle z-y,\alpha\rangle^2} dy\\
&\geq& k^2(\alpha) \int_{C_z^r-z}G(0,y)\frac{1}{|y|^2} dy \\
&=& \infty.
\end{eqnarray*}
\end{remark}
It is easy to see  that for every $x\in D$ the map $f\to H_D^{\Delta_k}f(x)$ defines a positive Radon measure on $\R^d\setminus D$. We denote this measure by $H_D^{\Delta_k}(x,dy)$. The following results are obtained in a convenient way by using formula (\ref{solhar}) of the above theorem.
\begin{corollary}
 For every $x\in D$, $H_D^{\Delta_k}(x,dy)$ is a probability measure supported by
$$
\partial D\cup \left(\cup_{\alpha\in R_+}\sigma_\alpha(D)\right)
$$
and satisfies
$$
\frac{\sqrt{w_k(x)}}{\sqrt{w_k(y)}}H_D^{\Delta_k}(x,dy)=H_D^{\Delta-q}(x,dy)+\sum_{\alpha\in R_+}\frac{|\alpha|^2k(\alpha)}{\langle y,\alpha\rangle^2}G_D^{\Delta-q}(x,\sigma_\alpha y) dy.
$$
\end{corollary}

\begin{corollary}
Let $D$ be a $\Delta$-regular bounded open set such that $\overline{D}$ is in some Domain of $\R^d\setminus\cup_{\alpha\in R_+}H_\alpha$.
Let $f$ be a continuous function on $\partial D\cup \left(\cup_{\alpha\in R_+}\sigma_\alpha(D)\right)$. If $f$ is locally H\"{o}lder continuous on $\cup_{\alpha\in R_+}\sigma(D)$ then $H_D^{\Delta_k}f\in C^2(D)$ and, for every $x\in D$,
$$
\Delta_k\left(H_D^{\Delta_k}f\right)(x)=0.
$$
\end{corollary}

\begin{proof}
Since $H_D^{\Delta-q}(f\sqrt{w_k})$ is a solution of the Schr\"{o}dinger's equation (\ref{shr}), the hypoellipticity of the operator $\Delta-q$ on $D$ implies that $H_D^{\Delta-q}(f\sqrt{w_k})\in C^\infty(D)$. Moreover, since $Nf$ is locally H\"{o}lder continuous on $D$, $G_D^\Delta\left(\sqrt{w_k}Nf\right) \in C^2(D)$  and consequently $G_D^{\Delta-q}\left(\sqrt{w_k}Nf\right)\in C^2(D)$. Then it follows from (\ref{solhar}) that $H_D^{\Delta_k}f\in C^2(D)$. For every $\varphi\in C^\infty_c(D)$, direct computation using  (\ref{ddk}) yields
$$
 \int \Delta_k\left(H_D^{\Delta_k}f\right)(x)\varphi(x) w_k(x) dx=\int H_D^{\Delta_k}f(x)\Delta_k\varphi(x) w_k(x) dx.
$$
This completes the proof.
\end{proof}

Let $D$ be an open subset of $\R^d$. The operator $\Delta_k$ is said to be hypoelliptic on $D$ if, for every $f\in C^\infty(D)$, every continuous function $h$ on $\R^d$ which satisfies
$$
\int_{\R^d}h(x)\Delta_k\varphi(x) w_k(x) dx =\int  f(x)\varphi(x) w_k(x) dx \quad \textrm{ for every }\; \varphi \in C^\infty_c(D)
$$
is infinitely differentiable on $D$. We note that the problem of the hypoellipticity of $\Delta_k$  is discussed in \cite{kh,hmkt}, where the authors show that  $\Delta_k$ is hypoelliptic on $D$ provided  $D$ is  invariant under all reflections $\sigma_\alpha$. However, if $D$ is not invariant, the question whether $\Delta_k$ is hypoelliptic on $D$ or not remaind open.

\begin{theorem}
Let $D$ be a $\Delta_k$-regular open set. Then $\Delta_k$ is  hypoelliptic on $D$ if and only if $D$ is invariant.
\end{theorem}
\begin{proof}
It is obviously sufficient to prove that $\Delta_k$ is not hypoelliptic on $D$ provided $D$ is not invariant. Assume that $D$ is not invariant. Since
 the open set $D\setminus\cup_{\alpha\in R_+}H_\alpha$ is also not invariant, there exists a nonempty  open ball $B$
such that
$$
\overline{B}\subset D\setminus\cup_{\alpha\in R_+}H_\alpha\quad\textrm{and}\quad\sigma_\alpha(B)\subset \R^d\setminus D\;\textrm{ for some }\;\alpha\in R_+.
$$
 We also choose the ball $B$ small enough such that, for every $\alpha\in R_+$,
$$
\sigma_\alpha(B)\subset D \quad \textrm{ or } \quad \sigma_\alpha(B)\subset\R^d\setminus D.
$$
Let $I :=\{\alpha\in R_+;\; \sigma_\alpha(B)\subset\R^d\setminus D\}$ and $J :=R_+\setminus I$.
Let $f$ be a continuous function on $\R^d\setminus D$ and denote $H_D^{\Delta_k}f$ by $h$.
Since $B$ is $\Delta$-regular and $h$ satisfies
$$
\int h(x)\Delta_k\varphi(x) w_k(x) dx=0 \quad \textrm{ for every}\; \varphi\in C^\infty_c(B),
$$
 it follows from Theorem \ref{t1} that $B$ is $\Delta_k$-regular and, for every $x\in B$,
\begin{equation}\label{hyp}
h(x)=\frac{1}{\sqrt{w_k(x)}}\left(H_B^{\Delta-q}(h\sqrt{w_k})(x)+G_B^{\Delta-q}\left(\sqrt{w_k}Nh\right)(x)\right).
\end{equation}
Let $g_1$ and $g_2$ be the functions defined on $B$ by

$$
g_1(x)=\sum_{\alpha\in J}\frac{|\alpha|^2k(\alpha)}{\langle x,\alpha\rangle^2}h(\sigma_\alpha x)\quad
\textrm{and}
\quad
g_2(x)=\sum_{\alpha\in I}\frac{|\alpha|^2k(\alpha)}{\langle x,\alpha\rangle^2}f(\sigma_\alpha x).
$$
It is clear that the function $g_2$ is not trivial and $Nh=g_1+g_2$.
Now, assume that $h\in C^\infty(D)$. Then $g_1\in C^\infty(B)$ and therefore $G_B^{\Delta-q}\left(\sqrt{w_k}g_1\right)\in C^\infty(B)$. Furthermore, since $H_B^{\Delta-q}(h\sqrt{w_k})\in C^\infty(B)$, it follows from  (\ref{hyp}) that $G_B^{\Delta-q}\left(\sqrt{w_k}g_2\right)\in C^\infty(B)$. Thus $-(\Delta-q)G_B^{\Delta-q}\left(\sqrt{w_k}g_2\right)=\sqrt{w_k}g_2\in C^\infty(B)$ and therefore $g_2\in C^\infty(B)$, a contradiction.
Hence $h$ is not infinitely differentiable on $D$ and consequently the Dunkl Laplacian $\Delta_k$ is not hypoelliptic on $D$.
\end{proof}


\begin{thebibliography}{X-XX1}
\bibitem{mbke}M. Ben Chrouda and K. El Mabrouk, Dirichlet problem associated with Dunkl Laplacian on $W$-invariant open sets, \textit{Preprint. arxiv: 1402.1597} (2014).
\bibitem{abwhhh}A. Boukricha, W. Hansen and H. Hueber, Continuous solutions of the generalized
Schr\"{o}dinger equation and perturbation of harmonic spaces, \textit{Expo. Math.} \textbf{5} (1987)
97--135.
\bibitem{dunkl1}C. F. Dunkl, Differential-difference operators associated to reflection groups, \textit{Trans. Am. Math. Soc.}
\textbf{311} (1989) 167--183.
 \bibitem{kczz}K. L. Chung and Z. Zhao, \textit{From Brownian motion to Schr\"{o}dinger's equation}, Springer-Verlag, 1995.
 \bibitem{kh}K. Hassine, Mean value property associated with the Dunkl Laplacian, \textit{Preprint. arxiv: 1401.1949} (2014).
 \bibitem{hmkt} H. Mejjaoli and K.  Trimèche, Hypoellipticity and hypoanalyticity of the Dunkl Laplacian operator, \textit{Integral Transforms  Spec. Funct.} \textbf{15} (2004) 523--548.
\end{thebibliography}
\end{document}